\documentclass{article}
\usepackage{amsmath,amsthm,amssymb}
\usepackage{amsfonts}
\usepackage{eucal}

\newcommand{\tpt}{$({\mathbf{2+2}})$} 
\newcommand{\lessp}{\prec} 
\newcommand{\morep}{\succ}

\newcommand{\cP}{\mathcal{P}}%primitive posets
\newcommand{\cG}{\mathcal{G}}%general posets  
\newcommand{\cMp}{\mathcal{M}}%prim. matrices, not s-d
\newcommand{\cS}{\mathcal{S}}% prim. s-d
\newcommand{\cSp}{\mathcal{S}^{+}}%prim. self-d., with 1 in the corner
\newcommand{\cSn}{\mathcal{S}_0}%prim. s-d with zero diagonal
\newcommand{\cSd}{\mathcal{S}_1}%prim. s-d with non-zero diagonal

\newcommand{\rF}{\mathcal{R}}%primitive row-Fishburn matrices
\newcommand{\dual}[1]{\overline{#1}}          
\newcommand{\dg}{\mathrm{diag}}
\newcommand{\se}{\mathrm{se}}%internal SE-cells
\newcommand{\nw}{\mathrm{nw}}%internal NW-cells
\newcommand{\rs}{\mathrm{rs}}%reduced size

\newcommand{\mg}{\mathrm{mag}}%magnitude
\newcommand{\iso}{\mathrm{iso}}%isolated elts
\newcommand{\rest}{\mathrm{int}}%internal elts

\newcommand{\PP}{G} %gen. function of primitive posets
\newcommand{\GP}{G^*} %GF of general posets
\newcommand{\SP}{S^+} %gen. function of primitive pyramids
\newcommand{\SPP}{S} %gen. function of all primitive sym. posets
\newcommand{\SGP}{S^*} %GF  of s-d. gen. posets

\newcommand{\Sr}{S'} %GF for self-dual primitive, by reduced size
\newcommand{\Srn}{S'_0} %as above, restricted to zero-diagonal matrices
\newcommand{\Srd}{S'_1} %ditto non-zero diagonal

\newcommand{\cO}{{\cal O}} %O in asymptotic notation

\newtheorem{theorem}{Theorem}[section]
\newtheorem*{theorem*}{Theorem}
\newtheorem{corollary}[theorem]{Corollary}
\newtheorem*{corollary*}{Corollary}
\newtheorem{lemma}[theorem]{Lemma}
\newtheorem*{lemma*}{Lemma}

\newtheorem*{proposition*}{Proposition}
\newtheorem{conjecture}[theorem]{Conjecture}
\newtheorem*{conjecture*}{Conjecture}
\newtheorem{fact}[theorem]{Fact}
\theoremstyle{definition}

\newtheorem*{definition*}{Definition}

\newtheorem*{example*}{Example}
\newtheorem{problem}[theorem]{Problem}
\newtheorem*{problem*}{Problem}

\theoremstyle{remark}
\newtheorem{remark}[theorem]{Remark}

\title{Counting Self-Dual Interval Orders}
\date{}
\author{V\'\i t Jel\'\i nek\thanks{
This research was supported by grant no. 090038011 from the Icelandic
Research Fund and by grant Z130-N13 of the Austrian Research Foundation
(FWF). An extended abstract of this work will appear in the proceedings of
FPSAC 2011.}\\
\small
Faculty of Mathematics, University of Vienna,\\[-3pt] 
\small
Garnisongasse 3, 1090 Vienna, Austria.\\
\small
\texttt{jelinek@kam.mff.cuni.cz}
}
\begin{document}                         
\maketitle
\begin{abstract}
In this paper, we present a new method to derive formulas for the generating
functions of interval orders, counted with respect to their size, magnitude,
and number of minimal and maximal elements. Our method allows us not only to
generalize previous results on refined enumeration of general interval orders,
but also to enumerate self-dual interval orders with respect to analogous
statistics. 

Using the newly derived generating function formulas, we are able
to prove a bijective relationship between self-dual interval orders and
upper-triangular matrices with no zero rows. Previously, a similar bijective
relationship has been established between general interval orders and
upper-triangular matrices with no zero rows and columns.

\end{abstract}

\section{Introduction}
The aim of this paper is to enumerate interval orders (also known as
\tpt-free posets) with respect to several natural poset statistics,
including the size, the magnitude, and the number of minimal and maximal
elements. We are mostly motivated by the generating function formulas recently
obtained by Bousquet-M\'elou et al.~\cite{BMCDK}, Kitaev and
Remmel~\cite{KitaevRemmel}, and Dukes et al.~\cite{indistin}. 

Although the formulas derived in this paper provide a common generalization of
these previous results, the method we use is different. The previous
results were derived using a recursive bijection between interval orders and
ascent sequences, due to Bousquet-M\'elou et al.~\cite{BMCDK}. In this paper, we
instead use an encoding of interval orders by upper-triangular matrices without
zero rows and zero columns (we call such matrices \emph{Fishburn matrices}). Our
approach, which builds upon previous work of Haxell et al.~\cite{Haxell}, is
considerably simpler than the approach based on ascent sequences. More
importantly, our approach allows to easily capture the notion of poset
duality, which corresponds to transposition of Fishburn matrices. Consequently,
we are able to adapt our method to the problem of enumerating self-dual
interval orders, for which no explicit enumeration has been known before. As a
by-product, we establish a bijective correspondence between self-dual interval
orders and upper-triangular integer matrices with no zero rows, in which
several natural poset statistics map into natural matrix statistics.

\subsection*{Basic Notions}
All the posets considered in this paper are assumed to be finite. We also assume
that the posets are \emph{unlabeled}, that is, isomorphic posets are taken to be
identical. Let $P$ be a poset with a strict order relation $\lessp$. 
A \emph{strict down-set} of an element $y\in P$ is the set $D(y)$ of all the
elements of $P$ that are smaller than~$y$, i.e., $D(y)=\{x\in P;\; x\lessp y\}$.
Similarly, the \emph{strict up-set} of $y$, denoted by $U(y)$, is the set $\{x\in P;\; x\morep y\}$. Note
that $y$ is a minimal element of $P$ if and only if $D(y)$ is empty, and $y$ is
a maximal element if and only if $U(y)$ is empty.

For a poset $P$, the following conditions are known to be
equivalent~\cite{Bogart,FishburnPrvni}:
\begin{itemize}
 \item $P$ is \emph{\tpt-free}, that is, $P$ does not have an induced subposet
isomorphic to the disjoint union of two chains of length two.
\item $P$ has an \emph{interval representation}, that is, to each element
$x\in P$ we may associate a real closed interval $[l_x,r_x]$, in
such a way that $x\lessp y$ if and only if $r_x<l_y$.
\item For any two elements $x,y\in P$, the strict down-sets $D(x)$ and $D(y)$
are comparable by inclusion, i.e., $D(x)\subseteq D(y)$ or $D(y)\subseteq D(x)$.
\item For any two elements $x,y\in P$, the strict up-sets $U(x)$ and $U(y)$ are
comparable by inclusion.
\end{itemize}
The posets that satisfy these properties are known as \emph{interval orders} or
as \emph{\tpt-free posets}. Let us review some of their basic
properties. For a more thorough exposition, the reader is referred to
Fishburn's monograph~\cite{fishburn1985interval}. 

Let $P$ be an interval order. Two elements $x$ and $y$ of $P$ are
\emph{indistinguishable} if $U(x)=U(y)$ and $D(x)=D(y)$. This is an equivalence
relation on~$P$. If no two distinct elements of $P$ are indistinguishable, then
$P$ is said to be \emph{primitive}. Every interval order $P$ can be
uniquely obtained from a primitive interval order $P'$ by simultaneously
replacing each element of $P'$ by a positive number of `duplicates'. Thus, the
enumeration of primitive interval orders is a key step in the enumeration of
general interval orders.

Since any two strict down-sets in $P$ are comparable by inclusion, it is
possible to arrange all the distinct strict down-sets into an increasing chain
\[
D_1\subsetneq D_2 \subsetneq\dotsb\subsetneq D_m,
\]
where $m$ is the number of distinct strict down-sets determined by
elements of~$P$. An element $x\in P$ is said to have \emph{level $i$}, if
$D(x)=D_i$. Note that $D_1$ is always the empty set, and the elements of level
1 are exactly the minimal elements of~$P$. Following
Fishburn~\cite{FishburnLength,FishburnDM}, we call the number $m$ of distinct
strict down-sets the \emph{magnitude} of~$P$. It turns out that $m$ is also
equal to the number of distinct strict up-sets, and we can order the
strict up-sets of $P$ into a decreasing chain
\[
 U_1\supsetneq U_2 \supsetneq\dotsb\supsetneq U_m,
\]
and we say that $x$ has \emph{up-level $i$} if $U(x)=U_i$. The maximal
elements of $P$ are precisely the elements of up-level $m$, and we have
$U_m=\emptyset$. It can be shown~\cite{FishburnDM} that an element of level $i$
has an up-level greater than or equal to~$i$. An interval representation of
$P$ can be obtained by mapping an element $x$ with level $i$ and up-level
$j$ to the (possibly degenerate) interval $[i,j]$. This is the
unique representation of $P$ by intervals with endpoints belonging to the set
$[m]=\{1,2,\dotsc,m\}$, and in particular, there is no interval representation
of $P$ with fewer than $m$ distinct endpoints.

The \emph{dual} of a poset $P$ is the poset $\dual P$ with the same
elements as $P$ and an order relation $\dual\lessp$ defined by $x\dual\lessp
y\iff y\!\lessp\! x$. A poset is \emph{self-dual} if it is isomorphic to its
dual.
The dual of an interval order $P$ of magnitude $m$ is again an interval order of
magnitude $m$, and an element of level $i$ and up-level $j$ in $P$ has the
level $m+1-j$ and up-level $m+1-i$ in~$\dual P$.

Throughout this paper, an important part will be played by a bijective
correspondence between interval orders and a certain kind of integer matrices,
which we will call Fishburn matrices. We will state the key properties of the
correspondence without proof; more details can be found, e.g., in the work of
Fishburn~\cite{FishburnLength,fishburn1985interval}, where these matrices are
called `characteristic matrices'. 

A \emph{Fishburn matrix} is an upper-triangular square matrix
$M$ of nonnegative integers with the property that every row and every column
contains a nonzero entry. A Fishburn matrix is called \emph{primitive} if all
its entries are equal to 0 or~1. We will assume throughout this paper that each
matrix has its rows numbered from top to bottom, and columns numbered
left-to-right, starting with row and column number one. We let $M_{ij}$ denote
the entry of $M$ in row $i$ and column~$j$.

An interval order $P$ of magnitude $m$ corresponds to an $m\times m$ Fishburn
matrix $M$ with $M_{ij}$ being equal to the number
of elements of $P$ that have level $i$ and up-level~$j$. Conversely, given an
$m\times m$ Fishburn matrix $M$, we may recover the corresponding interval
order $P$ by taking the collection of intervals that contains precisely
$M_{ij}$ copies of the interval $[i,j]$, and taking this to be the
interval representation of~$P$. 

This correspondence is a bijection between Fishburn matrices and interval
orders. In fact, in this correspondence, each nonzero entry
$M_{ij}$ of $M$ can be associated with a set of $M_{ij}$ indistinguishable
elements of~$P$. Note that the sum of the $i$-th row of $M$ is equal to the
number of elements of level $i$ in $P$, and similarly for column-sums and
up-levels. 

Primitive interval orders correspond to primitive Fishburn matrices.
If the order $P$ is mapped to a matrix $M$, then the dual order $\dual P$ is
mapped to the matrix $\dual M$ obtained from $M$ by transposition along the
diagonal running from bottom-left to top-right. If a matrix $M$ is equal to
$\dual M$, we call it \emph{self-dual}. Of course, self-dual matrices are
representing precisely the self-dual interval orders.

\subsection*{Previous work and our results}

Interval orders are equinumerous with several other combinatorial 
structures. Apart from the correspondence between interval orders and Fishburn
matrices, there are also bijections mapping interval orders to ascent
sequences~\cite{BMCDK}, Stoimenow matchings~\cite{Stoimenow}, certain classes
of pattern-avoiding permutations~\cite{BMCDK,bivincular}, or special kinds of
inversion tables~\cite{Levande}. Some of these combinatorial
structures have been studied independently even before their relationship to
interval orders was discovered. Thus, some results on the enumeration of interval
orders were first derived in different contexts, and some of them
were in fact derived several times under different guises.

The concept of interval order has been introduced by
Fishburn~\cite{FishburnPrvni} in 1970. In 1976, Andresen and
Kjeldsen~\cite{AndresenKjeldsen}, motivated by a counting problem related to
subgraphs of transitively oriented tournaments, introduced (under different
terminology) the problem of enumerating Fishburn matrices. They studied,
among other problems, the number of primitive Fishburn matrices with respect to
their dimension and the number of elements in the first row, which in poset
terminology corresponds to the number of primitive interval orders of a given
magnitude and number of minimal elements (but not the number of all elements).
Andresen and Kjeldsen obtained asymptotic bounds for the number of these
matrices, as well as recurrence formulas that allowed them to compute several
exact initial values. At the time of their writing, the connection between
Fishburn matrices and interval orders was not known, and it appears that their
results went unnoticed by later works on interval orders. 

In 1987, Haxell, McDonald and Thomason~\cite{Haxell} provided an efficient
way to compute the number of interval orders, using a recurrence derived using
Fishburn matrices, which were already known to be equinumerous with interval
orders, thanks to the work of Fishburn~\cite{fishburn1985interval}.

In 1998, Stoimenow~\cite{Stoimenow} introduced the concept of `regular
linearized chord diagram', later often referred to as a `Stoimenow
matching'. A Stoimenow matching of size $n$ is a matching on the set $[2n]$ in 
which no two nested edges have adjacent endpoints. Stoimenow has
introduced these matchings as a tool in the study of Vassiliev invariants of
knots, and computed several asymptotic bounds on their number. Later, these
bounds were improved by Zagier~\cite{Zagier}, who also showed that
the generating function of Stoimenow matchings enumerated by their size admits a
simple formula 
\begin{equation}
F(x)=\sum_{n\ge 0}\prod_{i=1}^n\left(1-(1-x)^i\right).\label{eq-zag}
\end{equation}

Recently, Bousquet-M\'elou, Claesson, Dukes and Kitaev~\cite{BMCDK} have found
a sequence of bijections, showing that interval orders are equinumerous with
several other combinatorial objects, including Stoimenow matchings and ascent
sequences. They have also provided an alternative proof for \eqref{eq-zag}, and
derived a formula for the refined generating function that counts interval order
with respect to their size and magnitude. These results have prompted a renewed
interest in the study of interval orders. Several other papers have focused on
bijections between interval orders and other objects. For
instance, Dukes and Parviainen~\cite{dukes2010ascent} have described a direct
bijection between Fishburn matrices and ascent sequences, Claesson and
Linusson~\cite{CL} gave a direct mapping from Fishburn matrices to
Stoimenow matchings, while the papers of Claesson et al.~\cite{CDK} and Dukes
et al.~\cite{DJK} extend the bijection between interval orders and
Fishburn matrices to more general combinatorial structures.

Another line of research has focused on refined enumeration of interval orders
with respect to some natural poset statistics. Dukes, Kitaev, Remmel and
Steingr\'\i msson~\cite{indistin} have found an expression for the generating
function that enumerates primitive interval orders with respect to their size
and magnitude, and deduced a formula that counts interval orders by their size,
magnitude and the number of indistinguishable elements. Kitaev and
Remmel~\cite{KitaevRemmel} have obtained, among other results, the formula 
\begin{equation}
F(x,y)= 1+\sum_{n\ge
0}\frac{xy}{(1-xy)^{n+1}}\prod_{i=1}^n\left(1-(1-x)^i\right),\label{eq-kitrem}
\end{equation}
where $F(x,y)$ is the generating function of interval orders in which $x$ counts
the size and $y$ the number of minimal elements of the interval order. They
conjectured that $F(x,y)$ can be also expressed in the following form:
\begin{equation}
 F(x,y)=\sum_{n\ge
0}\prod_{i=1}^n\left(1-(1-x)^{i-1}(1-xy)\right). \label{eq-conj}
\end{equation}
This conjecture has been subsequently confirmed by Yan~\cite{Yan} and
independently by Levande~\cite{Levande}. Let us remark formula~\eqref{eq-conj}
also appears in Zagier's work~\cite[Theorem 1]{Zagier}, but there it is
interpreted in terms of Stoimenow matchings, not interval orders. 

In Section~\ref{sec-nonsym} of this paper, we generalize the
above-mentioned results of Bousquet-M\'elou et al.~\cite{BMCDK}, Dukes et
al.~\cite{indistin}, and Kitaev and Remmel~\cite{KitaevRemmel}, by obtaining a
closed-form expression for the generation function of primitive interval orders,
counted with respect to their magnitude, their size, and their number of minimal
and maximal elements. From this expression, it is possible to directly derive
the generating function of general interval orders, or of interval orders with
bounded size of indistinguishability classes, counted with respect to the same
statistics. 

However, the main significance of our results is not in counting more
statistics than previous papers, but rather in presenting a much simpler method
to derive the generating functions. Previous results were largely based on
a bijection, constructed by Bousquet-M\'elou et al.~\cite{BMCDK}, which maps
interval orders to a certain kind of integer sequences, known as `ascent
sequences'. This bijection has a complicated recursive definition, which
consequently leads to difficult recurrences for generating functions, which then
require a large amount of ingenuity to be solved into closed form expressions,
typically by an application of the kernel method. In contrast, the arguments in
this paper are based on the direct encoding of interval orders as Fishburn
matrices. We exploit a relationship between Fishburn matrices of dimension $m$
and those of dimension $m+1$ to obtain a recurrence for the generating
function that can be easily solved by elementary expression manipulation.

To further illustrate the benefit of this approach, in Section~\ref{sec-sym} we
enumerate self-dual interval orders, by a slightly more elaborate application
of the same basic technique. The duality map of interval orders translates into
`obvious' involutions on related combinatorial classes --- for example, it
corresponds to transposition of a Fishburn matrix along the north-east diagonal,
or to the reversal of vertex order in a Stoimenow matching. This suggest that
poset duality represents a fundamental symmetry of these classes of objects, and
it is therefore natural to consider the enumeration of symmetric objects, that
is, of objects that are fixed points of this symmetry map.

The problems of counting self-dual interval orders with respect to their size
and of counting primitive self-dual Fishburn matrices with respect to their
dimension do not seem to have been addressed by previous research, presumably
because there is no known way to characterize self-duality in terms of ascent
sequences. In view of this, it is remarkable that the expressions we obtain for
the generating functions of self-dual objects are almost as simple as those of
their non-self-dual counterparts. 

Moreover, in Section~\ref{sec-red}, we introduce \emph{row-Fishburn matrices},
which are upper-triangular matrices with no zero rows, and we prove that there
is a close link between the refined enumeration of these matrices and the
refined enumeration of self-dual interval orders. This yields a surprising
analogue to the correspondence between general interval orders and Fishburn
matrices.

\section{Enumeration of Interval Orders}\label{sec-nonsym}

Recall that a primitive poset is a poset that does not contain any two
indistinguishable elements. Our main concern will be to find an
expression for the generating function of primitive interval orders, or
equivalently, of 01-Fishburn matrices.

Let us call an element of a poset $P$ \emph{isolated} if it is not comparable to
any other element of $P$. Note that an element is isolated if and only if it is
both minimal and maximal. The number of isolated elements of an interval order
$P$ is equal to the value of the top-right cell of the corresponding Fishburn
matrix. We will call the top-right cell of the matrix \emph{the corner cell}.

Let us also say that an element of the poset $P$ is
\emph{internal} if it is neither minimal nor maximal. We will consider these
statistics of an interval order $P$: 
\begin{itemize}
\item the magnitude of $P$, denoted by $\mg(P)$,
\item the number of isolated elements, denoted by $\iso(P)$,
\item the number of non-isolated minimal elements, denoted by $\min(P)$,
\item the number of non-isolated maximal elements, denoted by $\max(P)$, and
\item the number of internal elements, denoted by
$\rest(P)$. 
\end{itemize}
In particular, the size of $P$ is equal to $\iso(P)+\min(P)+\max(P)+\rest(P)$,
and its number of minimal elements is equal to $\min(P)+\iso(P)$. In
Table~\ref{tab-stats}, we summarize these statistics, with their interpretation
both in terms of posets and in terms of matrices. As a matter of convenience,
we restrict ourselves to non-empty interval orders and non-empty matrices in
all the arguments. If $M$ is the Fishburn matrix representing an interval order
$P$, we write $\iso(M)$ as a synonym for $\iso(P)$, and similarly for the other
statistics.

Let $\cP$ be the set of all non-empty primitive interval
orders, and let $\PP$ be the generating function
\[
 \PP(t,v,w,x,y)=\sum_{P\in\cP} t^{\mg(P)}x^{\iso(P)} y^{\min(P)} v^{\max(P)}
w^{\rest(P)}.
\]

We shall use the following notation: for an integer $n\ge 0$,
we let $V_n(a,b)$ denote the polynomial $(a+1)(b+1)^n-1$.

\begin{table}
\small\raggedright
\begin{tabular}{|c|p{0.32\textwidth}|p{0.32\textwidth}|c|}
\hline
\textbf{statistic} & \textbf{poset interpretation} & \textbf{matrix
interpretation} & \textbf{variable} \\
\hline\hline
$\mg$ & magnitude & number of rows & $t$ \\ \hline
$\iso$ & num. of isolated elements & value of corner cell & $x$\\ \hline 
$\min$ & num. of non-isolated minimal elements & sum of
the first row, except the corner cell & $y$\\ \hline
$\max$ & 
num. of non-isolated maximal elements\hfil & 
sum of the last column, except the corner cell\hfil &
 $v$\\ \hline
$\rest$ & num. of internal elements & sum of cells outside
first row and last column & $w$\\ \hline
\end{tabular}
\caption{The statistics of interval orders}\label{tab-stats}
\end{table}

We can now state the main result of this section.
\begin{theorem}\label{thm-nesym}
 The generating function $\PP(t,v,w,x,y)$ satisfies the identity
\begin{align}  
\PP(t,v,w,x,y) =\sum_{n\ge 
0}t^{n+1}\frac{V_n(x,y)}{1+tV_n(v,w)}\prod_{i=0}^{n-1}\frac{V_i(v,w)
} { 1+tV_i(v,w)}.  \label{eq-main}
\end{align}
\end{theorem}                                                                  

\begin{remark}\label{rem-converg}
Let $S_n\equiv S_n(t,v,w,x,y)$ denote the $n$-th summand of the sum on the right-hand side of~\eqref{eq-main}. Clearly, $S_n$ is a
multiple of $t^{n+1}$. Consequently, the sum on the right-hand side of~\eqref{eq-main} converges as a sum of power 
series in~$t$. Moreover, since for each $i\ge 0$ the polynomial $V_i(v,w)$ has constant term equal
to zero, it follows that $S_n$ has total degree at least $n$ in the
variables $v$ and~$w$. Thus, the sum also converges as a sum of power series in
$v$ and~$w$. Furthermore, for all $k$, the coefficient 
of $t^k$ in $S_n$ is a polynomial in $v$, $w$, $x$ and $y$, and for all
$k,\ell$, the coefficient of $v^kw^\ell$ is a polynomial in $t$, $x$, and~$y$.
\end{remark}

The properties stated in the above remark make the identity~\eqref{eq-main}
amenable to many combinatorially meaningful substitutions. Before we state the
proof of the theorem, we demonstrate several possible substitutions. Note
that some of the formulas we derive have been previously obtained by
different methods.

\begin{corollary}\label{cor-primvel}\cite[Theorem 8]{indistin}
Let $p_k$ be the number of primitive interval orders of size~$k$. The generating function of $p_k$ is equal to 
\begin{align*}
\sum_{k\ge 1} p_k x^k &= \PP(1,x,x,x,x)=\sum_{n\ge 0} \prod_{i=0}^{n} \frac{ V_i(x,x) }{1+V_i(x,x)}=\sum_{n\ge 0} \prod_{i=0}^{n}\left(1- \frac{1}{(1+x)^{i+1}}\right).\\
\end{align*}
\end{corollary}

\begin{corollary}\label{cor-primmag} Let $m_k$ be the number of primitive $k\times k$ Fishburn matrices (or equivalently, the number of primitive interval orders of magnitude $k$). Then
\begin{align*}
\sum_{m\ge 1} m_k t^k &= \PP(t,1,1,1,1)= \sum_{n\ge 0}t^{n+1} \prod_{i=0}^n \frac{2^{i+1}-1}{1+t(2^{i+1}-1)}.
\end{align*}
\end{corollary}

Let $\cG$ be the set of non-empty interval orders, and let $\GP(t,v,w,x,y)$ be
the corresponding generating function
\[
 \GP(t,v,w,x,y)=\sum_{P\in\cG} t^{\mg(P)}x^{\iso(P)} y^{\min(P)} v^{\max(P)}
w^{\rest(P)}.
\]
Each interval order can be uniquely obtained from a primitive interval order by
replacing each element with a group of indistinguishable elements. In terms of
generating functions, this means that 
\[
\GP(t,v,w,x,y)= \PP(t,\frac{v}{1-v},\frac{w}{1-w},\frac{x}{1-x},\frac{y}{1-y}).
\]
By substituting into \eqref{eq-main}, and using, e.g., the identity
\[V_n(\frac{a}{1-a},\frac{b}{1-b})=-\frac{V_n(-a,-b)}{(1-a)(1-b)^n}=
-\frac{V_n(-a,-b)}{1+V_n(-a,-b)}\]
we get the next corollary.

\begin{corollary}\label{cor-gen} The generating function $\GP(t,v,w,x,y)$ is equal to
\begin{align*}
\sum_{n\ge 0} t^{n+1} \frac{ 1+V_n(-v,-w)}{ 1+V_n(-x,-y) }\cdot \frac{V_n(-x,-y)}{( t-1)V_n(-v,-w) -1 }\cdot
\prod_{i=0}^{n-1}\frac{V_i(-v,-w) }{ (t-1)V_i(-v,-w) -1}.
\end{align*}
\end{corollary}

From Corollary~\ref{cor-gen} we can obtain yet another proof of the
formula~\eqref{eq-conj} derived by Levande~\cite{Levande} and Yan~\cite{Yan} (and indirectly also by Zagier~\cite{Zagier}) for the generating function of
interval orders counted by their size and number of maximal elements.

\begin{corollary}\label{cor-gen-min}\cite{Levande,Yan,Zagier} 
Let $g_{k,\ell}$ be the number of interval orders with $k$ elements and having
exactly $\ell$ maximal elements (including isolated ones). We have
\begin{align*}
\sum_{k,\ell\ge 1} g_{k,\ell}\, r^k s^\ell&=\GP(1,rs,r,rs,r)=\sum_{n\ge 0} \prod_{i=0}^n -V_i(-rs,-r)\\
&=\sum_{n\ge 0} \prod_{i=0}^n (1-(1-rs)(1-r)^i).
\end{align*}
\end{corollary}

Kitaev and Remmel~\cite{KitaevRemmel} have obtained a different expression for
the generating
function from the previous corollary. This alternative expression can also be
derived from the general formula for $\GP$.

\begin{lemma}\label{lem-gen-min2}\cite{KitaevRemmel} With $g_{k,\ell}$ as in Corollary~\ref{cor-gen-min}, we have
\begin{equation}
\sum_{k,\ell\ge 1} g_{k,\ell}\, r^k s^\ell 
=\sum_{n\ge 0} \frac{rs}{(1-rs)^{n+1}}\prod_{i=0}^{n-1}\left(1-(1-r)^{i+1}\right).\label{eq-gen-min2}
\end{equation}
\end{lemma}
\begin{proof}
Since the dual of an interval order is also an interval order, $g_{k,\ell}$ is also equal to the number of interval orders with $k$ elements and $\ell$ minimal elements. Therefore, $\sum_{k,\ell} g_{k,\ell} r^k s^\ell=\GP(1,r,r,rs,rs)$, which is equal to
\begin{align*}
&\hphantom{=}\sum_{n\ge0}\!
\frac{1+V_n(-r,-r)}{1+V_n(-rs,-rs)}\left(-V_n(-rs,-rs)\right)\prod_{i=0}^{n-1}
\left(-V_i(-r,-r)\right) \displaybreak[0]\\
&=\!\sum_{n\ge0}\!
\frac{1-(1-rs)^{n+1}}{(1-rs)^{n+1}}\left(1+V_n(-r,-r)\right)\prod_{i=0}^{n-1}
\left(-V_i(-r,-r)\right) \displaybreak[0]\\
&=\!\sum_{n\ge0}\!
\frac{1-(1-rs)^{n+1}}{(1-rs)^{n+1}}\!\prod_{i=0}^{n-1}\left(-V_i(-r,-r)\right) -
\!\sum_{n\ge0}\! \frac{1-(1-rs)^{n+1}}{(1-rs)^{n+1}}
\prod_{i=0}^n\left(-V_i(-r,-r)\right)\\
&=\!\sum_{n\ge0}\!\left(\frac{1-(1-rs)^{n+1}}{(1-rs)^{n+1}}-\frac{1-(1-rs)^{n}}{
(1-rs)^{n}}\right)\prod_{i=0}^{n-1}\left(-V_i(-r,-r)\right),\\
\end{align*}
which simplifies to yield~\eqref{eq-gen-min2}.
\end{proof}

Let us now prove Theorem~\ref{thm-nesym}. Define
$\PP_k(v,w,x,y)=[t^k]\PP(t,v,w,x,y)$, that is, $\PP_k$ is the coefficient of
$t^k$ in $\PP$. We will state the proof in the language of
Fishburn matrices rather than in the equivalent language of interval orders. It
is thus convenient to view $\PP_k$ as the generating function of the primitive
$k\times k$ Fishburn matrices.

The next lemma provides the main idea in the proof of Theorem~\ref{thm-nesym}.

\begin{lemma}\label{lem-rekur}
For any $k\ge 1$, we have
\begin{equation}
\PP_{k+1}(v,w,x,y)=v\PP_k(v+w+vw,w,x+y+xy,y)-v\PP_k(v,w,x,y).\label{eq-rek}
\end{equation}
\end{lemma}
\begin{proof}
Let $\cMp_k$ denote the set of primitive $k\times k$ Fishburn matrices. We will 
describe an operation which from a given matrix $M\in\cMp_k$ produces a 
(typically not unique) new matrix $M'\in \cMp_{k+1}$.
The matrix $M'$ is created by adding to $M$ a new rightmost column and a new bottom row, and filling them according to these rules:
\begin{itemize}
\item $M'_{k+1,k+1}=1$, and all the other cells in row $k+1$ of $M'$ have value~0.
\item 
For every $j\le k$, if $M_{j,k}=0$, then $M'_{j,k}=M'_{j,k+1}=0$.
\item 
For every $j\le k$, if $M_{j,k}=1$ we choose one of the three possibilities
to fill $M'_{j,k}$ and $M'_{j,k+1}$: either $M'_{j,k}=0$ and
$M'_{j,k+1}=1$, or $M'_{j,k}=M'_{j,k+1}=1$, or $M'_{j,k}=1$ and $M'_{j,k+1}=0$. 
\item 
Any other cell has the same value in $M'$ as in~$M$.
\end{itemize}                                                              
If $M$ has $p$ 1-cells in column $k$, then the above operation can produce 
$3^p$ matrices~$M'$. All these $3^p$ matrices are upper-triangular, all of 
them have at least one 1-cell in each row, all of them have at least one 
1-cell in each column different from column $k$, and all except for exactly one of them have at least one 1-cell in column~$k$. In other words, for a given $M\in \cMp_k$ 
with $p$ 1-cells in column~$k$, the above operation produces $3^p-1$ matrices $M'$ 
from $\cMp_{k+1}$ (and one `bad' matrix not belonging to $\cMp_{k+1}$). It is not difficult to see that each matrix $M'\in \cMp_{k+1}$ can be 
created in this way from exactly one matrix $M\in \cMp_k$.

More generally, suppose that $M\in\cMp_k$ is a matrix with $\iso(M)\!=\!a$,
$\min(M)\!=\!b$, $\max(M)=c$ and $\rest(M)=d$, that is, $M$ contributes the
monomial $x^ay^bv^cw^d$ into $\PP_k(v,w,x,y)$. Then all the Fishburn matrices
produced from $M$ have generating function $v\left((x+y+xy)^a y^b (v+w+vw)^c
w^d-x^ay^bv^cw^d \right)$, where the leftmost factor of $v$ counts the 1-cell
$(k+1,k+1)$ of $M'$. Summing this expression over all $M\in \cMp_k$ gives the
recurrence from the statement of the lemma.
\end{proof}  
    
We remark the recursive procedure that generates matrices of $\cMp_{k+1}$ from
matrices of $\cMp_k$ is not new. In fact, this idea has already been used by
Haxell, McDonald, and Thomason~\cite{Haxell} to obtain
an efficient algorithm for the enumeration of interval orders. It is also very
closely related to the approach that Khamis~\cite{Khamis} has used to derive a
recurrence formula for the number of interval orders of a given size and height.

We now deduce Theorem~\ref{thm-nesym} from Lemma~\ref{lem-rekur} by simple
manipulation of series. Let us introduce the following notation: for any power
series $F(v,w,x,y)$, let $\sigma[F(v,w,x,y)]$ denote the power series 
$F(v+w+vw,w,x+y+xy,y)$, that is, $\sigma$ represents the substitution of
$v+w+vw$ for~$v$ and $x+y+xy$ for~$x$. Let $\sigma^{(i)}[F(v,w,x,y)]$ denote the
$i$-fold iteration 
of~$\sigma$; in other words, $\sigma^{(0)}[F]=F$ and for $i\ge 1$, 
$\sigma^{(i)}[F]=\sigma\left[\sigma^{(i-1)}[F]\right]$. It can be 
easily checked by induction that 
\[
\sigma^{(i)}[F(v,w,x,y)]=F(V_i(v,w),w,V_i(x,y),y).
\]

Writing $\PP_k$ for $\PP_k(v,w,x,y)$ and $\PP$ for $\PP(t,v,w,x,y)$, we can
rewrite \eqref{eq-rek} as $\PP_{k+1}=v\sigma[\PP_k]-v\PP_k$. Multiplying this by
$t^{k+1}$ and summing for all $k\ge 1$, we get
$
\PP-t\PP_1=tv\sigma[\PP]-tv\PP.
$ 
Since $\PP_1=x$, this simplifies into
\begin{equation}
\PP=\frac{tx}{1+tv}+\frac{tv}{1+tv}\sigma[\PP].\label{eq-rek2}
\end{equation}          

Substituting the right-hand side of this expression for the 
occurrence of $\PP$ on the right-hand side, we obtain
\begin{equation}
\PP=\frac{tx}{1+tv}+\frac{tv}{1+tv}\sigma\left[\frac{tx}{1+tv}\right]
+\frac{tv}{1+tv}\sigma\left[\frac{tv}{1+tv}\right]\sigma^{(2)}[\PP].\label{eq-rek3}
\end{equation}
We may again substitute the right-hand side of~\eqref{eq-rek2} into the right-hand side of~\eqref{eq-rek3}. In general, iterating this substitution $m$ times gives the identity
\begin{equation}
\PP=\left(\sum_{n=0}^m \sigma^{(n)}\left[\frac{tx}{1+tv}\right]\prod_{i=0}^{n-1}\sigma^{(i)}\left[\frac{tv}{1+tv}\right]\right)
+ \left(\prod_{i=0}^m\sigma^{(i)}\left[\frac{tv}{1+tv}\right]\right)\sigma^{(m+1)}[\PP].\label{eq-rekm}
\end{equation}
Since the rightmost summand of the right-hand side of~\eqref{eq-rekm} is $\cO(t^{m+1})$, we can take the limit as $m$ goes to infinity, to obtain
\begin{align*}
\PP&=\sum_{n\ge 0} \sigma^{(n)}\left[\frac{tx}{1+tv}\right]\prod_{i=0}^{n-1}\sigma^{(i)}\left[\frac{tv}{1+tv}\right]\\
%&=\sum_{n\ge 0}
%t^{n+1}\sigma^{(n)}\left[\frac{x}{1+tv}\right]\prod_{i=0}^{n-1}\sigma^{(i)}
%\left [\frac{v}{1+tv}\right]\\
&=\sum_{n\ge 0} t^{n+1}\frac{V_n(x,y)}{1+tV_n(v,w)}\prod_{i=0}^{n-1}\frac{V_i(v,w)}{1+tV_i(v,w)}.
\end{align*}
This proves Theorem~\ref{thm-nesym}.

\section{Self-Dual Interval Orders}\label{sec-sym}

Recall that a $k\times k$ Fishburn matrix $M$ represents a self-dual interval
order if and only if $M$ is invariant under transposition along the north-east
diagonal, or in other words, if for each $i,j$ the cell $(i,j)$ has the same
value as the cell $(k-j+1,k-i+1)$. We will say that the two cells $(i,j)$ and
$(k-j+1,k-i+1)$ form a \emph{symmetric pair}.

As in the previous section, we will first concentrate on enumerating the
primitive matrices, and the enumeration of general integer matrices is obtained as a consequence. Unless otherwise noted, the generating functions are only counting nonempty objects.

We distinguish 
three types of cells in a $k\times k$ matrix $M$: a cell $(i,j)$ is a
\emph{diagonal cell} if $i+j=k+1$, i.e., $(i,j)$ belongs to the north-east
diagonal of the matrix. If $i+j<k+1$ (i.e., $(i,j)$ is above and to the left of
the diagonal) then $(i,j)$ is a \emph{North-West cell}, or \emph{NW-cell},
while if $i+j>k+1$, then $(i,j)$ is an \emph{SE-cell}. The diagonal cells and 
SE-cells together uniquely determine a self-dual matrix. 

Apart of the statistics introduced in Section~\ref{sec-nonsym} (see
Table~\ref{tab-stats}), we will also consider three new statistics of a
matrix~$M$.
\begin{itemize}
\item $\dg(M)$ is the sum of all the diagonal cells except for the corner cell.
\item $\se(M)$ is the sum of all the SE-cells that do not belong to the last
column.
\item $\nw(M)$ is the sum of all the NW-cells that do not belong to the first
row.
\end{itemize}
In particular, the sum of all cells in $M$ is equal to
$\iso(M)+\min(M)+\max(M)+\se(M)+\nw(M)+\dg(M)$. In a self-dual matrix we of
course have $\min(M)=\max(M)$ and $\se(M)=\nw(M)$, so the above sum is equal to
$\iso(M)+\dg(M)+2\max(M)+2\se(M)$.

Notice that if $k>1$, then among all the $k\times k$ primitive self-dual
Fishburn matrices, exactly half have the corner cell filled with 1 and half have
the corner cell filed with 0. This is because changing the corner cell from 1 to
0 cannot create an empty row or empty column, since 
the cells $(1,1)$ and $(k,k)$ are always 1-cells and the value of the corner 
cell also does not affect the symmetry of the matrix. It
is simpler to first enumerate the symmetric matrices whose 
corner cell has the fixed value 1, and then use this result to 
obtain the full count, rather than to enumerate all the symmetric matrices at once.
 
Let $\cSp$ be the set of primitive self-dual Fishburn matrices whose corner
cell is equal to~1. Define the generating function $\SP(t,v,w,z)$ by
\[
\SP(t,v,w,z)=\sum_{M\in\cSp} t^{\mg(M)} v^{\max(M)} w^{\se(M)} 
z^{\dg(M)}.
\]  
The next theorem is the key result of this section.
\begin{theorem}\label{thm-sym} The generating function $\SP(t,v,w,z)$ is equal
to
\begin{equation}                                                   
\sum_{n\ge 0} t^{2n+1}
\frac{1+tV_n(v,w)}{1+t^2V_n(v,w)}(1+z)^n(1+v)^n
(1+w)^{\binom{n}{2}}\prod_{i=0}^{n-1}
\frac{V_i(v,w)}{1+t^2V_i(v,w)}.\label{eq-sym}
\end{equation}  
\end{theorem}
The comments in Remark~\ref{rem-converg} apply to the
expression~\eqref{eq-sym} as well.

Before we prove Theorem~\ref{thm-sym}, we first state some of its consequences.
Although Theorem~\ref{thm-sym} provides all the information we need for our
enumerations, it is often more convenient to work with the closely related
generating function that counts all Fishburn matrices, rather than just those
that belong to~$\cSp$. Let $\cS$ be the set of primitive self-dual Fishburn
matrices, and define
\[
 \SPP(t,v,w,x,z)=\sum_{M\in\cS} t^{\mg(M)} v^{\min(M)+\max(M)}
w^{\se(M)+\nw(M)} x^{\iso(M)} z^{\dg(M)}.
\]

\begin{lemma}\label{lem-spp}
 The generating function $\SPP$ satisfies the identity 
\begin{equation}
 \SPP(t,v,w,x,z)=(1+x)\SP(t,v^2,w^2,z)-t.\label{eq-spp}
\end{equation}
Consequently, $\SPP(t,v,w,x,z)$ is equal to
\begin{align*}
\begin{split}
  -t+(1+x)\sum_{n\ge 0} t^{2n+1}
\frac{1+tV_n(v^2,w^2)}{1+t^2V_n(v^2,w^2)}(1+z)^n(1+v^2)^n
(1+w^2)^{\binom{n}{2}}\\
\times\prod_{i=0}^{n-1}
\frac{V_i(v^2,w^2)}{1+t^2V_i(v^2,w^2)}.
\end{split}
\end{align*}
\end{lemma}
\begin{proof}
The factor $(1+x)$ on the right-hand side of~\eqref{eq-spp} corresponds to
the fact that each primitive self-dual matrix either belongs to $\cSp$ or is
obtained from a matrix in $\cSp$ by changing its corner cell from 1 to 0. The
subtracted $t$ accounts for the fact that the $1\times1$ matrix containing 0 is
not a Fishburn matrix, even though it can be obtained from a matrix in $\cSp$ by
changing the corner cell. The substitutions into $\SP$ correspond
straightforwardly to the fact that $v^{\min(M)+\max(M)}=v^{2\max(M)}$ for any
self-dual matrix, and similarly for~$w$.
\end{proof}

\begin{corollary}
 Let $s_m$ be the number of self-dual primitive interval orders on $m$ elements, with $s_0=1$. Then
\begin{align*}
 \sum_{m\ge 0} s_m x^m&=1+\SPP(1,x,x,x)=\sum_{n\ge
0}(1+x)^{n+1}\prod_{i=0}^{n-1}\left(
(1+x^2)^{i+1}-1\right).\label{eq-sym-vel}
\end{align*}
\end{corollary}

\begin{corollary}\label{cor-sym-mag}
 Let $r_m$ be the number of primitive self-dual $m\times m$ Fishburn matrices.
Then 
\begin{align*}
 \sum_{m\ge 1} r_m t^m&=\SPP(t,1,1,1)\\ 
&=-t+\sum_{n\ge 0}
2^{\binom{2n+2}{2}}t^{2n+1}\frac{1+(2^{n+1}-1)t}{1+(2^{n+1}-1)t^2}\prod_{i=0}
^{n-1}\frac{2^{i+1}-1}{1+(2^{i+1}-1)t^2}.
\end{align*}
\end{corollary}

Let $\SGP(t,v,w,x,z)$ be the generating function of (not necessarily primitive)
self-dual interval orders, with variables having the same meaning as in
$\SPP(t,v,w,x,z)$. Clearly, a Fishburn matrix $M$ representing a self-dual
interval order may be obtained in a unique way from a matrix $M'$ representing a
primitive self-dual interval order, by changing each diagonal 1-cell of $M'$
into an arbitrary non-zero cell, and by changing a symmetric pair of
non-diagonal 1-cells of $M'$ into a pair of nonzero cells having the same value.
Repeating the reasoning of Lemma~\ref{lem-spp}, we get the identity
\[
 \SGP(t,v,w,x,z)=\frac{1}{1-x}\SP\left(t,\frac{v^2}{1-v^2},\frac{w^2}{1-w^2},
\frac{z}{1-z}\right)-t.
\]
It follows that $\SGP(t,v,w,x,z)$ is equal to
\[
-t+\sum_{n\ge 0}\frac{t^{2n+1}\left( 1+(1-t)V_n(-v^2,-w^2) \right)\prod_{i=0}^{n-1}\frac{-V_i(-v^2,-w^2) }{1+( 1-t^2)V_i(-v^2,-w^2) }}{(1-x)(1-z)^n (1-v^2)^n(1-w^2)^{\binom{n}{2}}\left(1+(1-t^2)V_n(-v^2,-w^2) \right)}.
\]

\begin{corollary}\label{cor-sym-gen}
Let $g_m$ be the number of self-dual interval orders on $m$ elements, with $g_0=1$. Then
\begin{align*}
 \sum_{m\ge 0} g_m
x^m&=1+\SGP(1,x,x,x,x)\\
&=\sum_{n\ge 0}\frac{1}{(1-x^2)^{\binom{n+1}{2}}(1-x)^{n+1}}\prod_{i=0}^{n-1}
\left(1-(1-x^2)^{i+1}\right)\\
&=\sum_{n\ge0}\frac{1}{(1-x)^{n+1}}\prod_{i=0}^{n-1}\left(\frac{1}{(1-x^2)^{
i+ 1}} -1\right).
\end{align*}
\end{corollary}

The proof of Theorem~\ref{thm-sym} is based on the same general
idea as the proof of Theorem~\ref{thm-nesym}.
Let us define $\SP_k(v,w,z)=[t^k]\SP(t,v,w,z)$. The next lemma is the self-dual
analogue of Lemma~\ref{lem-rekur}.

\begin{lemma}\label{lem-symrekur}
For any $k\ge 1$, we have
\begin{equation}
\SP_{k+2}(v,w,z)=v(1+v)(1+z)\SP_k(v+w+vw,w,z)-v\SP_k(v,w,z).\label{eq-symrekur}
\end{equation}
\end{lemma}
\begin{proof}
Let $\cSp_k$ be the set of matrices of $\cSp$ of size $k\times k$. We 
will show how a given matrix $M\in\cSp_k$ can be extended into a matrix 
$M'\in\cSp_{k+2}$. Assume, just for the sake of this proof, that matrices in
$\cSp_k$ have rows and columns 
indexed by $1,2,\dotsc,k$, while matrices in $\cSp_{k+2}$ have rows and 
columns indexed by $0,1,\dotsc,k+1$. Thus, if a cell $(i,j)$ is a diagonal 
cell in~$M\in\cSp_k$, then $(i,j)$ is also a diagonal cell in~$M'\in\cSp_{k+2}$, and similarly for SE-cells and NW-cells.      

The matrix $M'$ is created from $M$ by adding a new left-most and right-most
row, and a new top-most and bottom-most column, and then filling the new cells
by these rules: 
\begin{itemize}
\item 
$M'_{k+1,k+1}=1$, and any
other cell in row $k+1$ of $M'$ has value~0.
\item 
$M'_{0,k+1}=1$. Note that the cell $(0,k+1)$ is the corner cell of $M'$.
\item
$M'_{1,k}$ is filled arbitrarily by 0 or 1, and $M'_{1,k+1}$ is filled arbitrarily by 0 or 1 as well. (Recall that $M_{1,k}=1$ by the definition of $\cSp$.)
\item For any $j\in\{2,\dotsc,k\}$, if $M_{j,k}=0$,
then $M'_{j,k}=M'_{j,k+1}=0$.         
\item For any $j\in\{2,\dotsc,k\}$, if $M_{j,k}=1$,
we choose one of three possibilities to fill $M'_{j,k}$ and $M'_{j,k+1}$: either $M'_{j,k}=0$ and
$M'_{j,k+1}=1$, or $M'_{j,k}=M'_{j,k+1}=1$, or $M'_{j,k}=1$ and $M'_{j,k+1}=0$.
\item Any other SE-cell or diagonal cell of $M'$ has the same value in $M'$ as
in $M$, and the NE-cells of $M'$ are filled in order to form a self-dual matrix.
\end{itemize}                                                       
From a given matrix $M\in\cSp_k$ with $\max(M)=p$, this procedure creates
$4\cdot 3^p$ distinct self-dual matrices of size 
$(k+2)\times (k+2)$. Of these $4\cdot 3^p$ matrices, all belong to 
$\cSp_{k+2}$ except for one matrix, which has column~$k$ (and hence also
row~$k$) filled with zeros.

In terms of generating functions, if $M$ is a matrix that 
contributes a monomial $v^pw^qz^r$ into the generating function $\SP_k(v,w,z)$, 
then the matrices in $\cSp_{k+2}$ created from $M$ have generating function

\begin{multline}
\sum_{\substack{M'\in\cSp_{k+2}\\M'\text{ obtained from }M}}
v^{\max(M')}w^{\se(M')}z^{\dg(M')}=\\
=v\bigl( (1+v)(1+z)(v+w+vw)^pw^qz^r-v^pw^qz^r\bigr).\label{eq-rek-lemma}
\end{multline}
It is easy to see that each matrix $M'\in\cSp_{k+2}$ can be generated by the
above rules from a unique matrix $M\in\cSp_{k}$.
By summing the identity \eqref{eq-rek-lemma} over all $M\in\cSp_k$, we obtain
the identity~\eqref{eq-symrekur}.
\end{proof}           

To prove Theorem~\ref{thm-sym} from Lemma~\ref{lem-symrekur}, we imitate the
proof of Theorem~\ref{thm-nesym} from Lemma~\ref{lem-rekur}. Let us write $\SP$
instead of $\SP(t,v,w,z)$, and $\SP_k$ instead of $\SP_k(v,w,z)$. Multiplying
\eqref{eq-symrekur} by $t^{k+2}$ and summing for all $k\ge 1$ gives 
\begin{equation*}
\SP-t\SP_1-t^2\SP_2=t^2v(1+v)(1+z)\sigma[\SP] - t^2v\SP.
\end{equation*}                                                                 
                
Since $\SP_1=1$ and $\SP_2=v$, this gives
\begin{equation}
\SP=\frac{t+t^2v}{1+t^2v}+\frac{t^2v(1+v)(1+z)}{1+t^2v}\sigma[\SP].
\label{eq-symrek2}
\end{equation}                                                                 
       
Repeatedly substituting for $\SP$ on the right-hand side of~\eqref{eq-symrek2},
and writing $V_n$ instead of $V_n(v,w)$ to save space,
we get for each $m\ge 0$ the identity
\begin{align*}
\SP&=\left(\sum_{n=0}^m \sigma^{(n)}\left[ \frac{t+t^2v}{1+t^2v}
\right]\prod_{i=0}^{n-1}
\sigma^{(i)}\left[\frac{t^2v(1+v)(1+z)}{1+t^2v}\right]\right)\\
 &\shoveright{\qquad\qquad+\prod_{i=0}^{m}
\sigma^{(i)}\left[\frac{t^2v(1+v)(1+z)}{1+t^2v}\right]\sigma^{(m+1)}[\SP]}\\
&=\sum_{n=0}^m t^{2n+1}\frac{1+tV_n}{1+t^2V_n}
\prod_{i=0}^{n-1}
\frac{(1+V_i)(1+z)V_i}{1+t^2V_i}+\cO(t^{2m+2})\\
&=\sum_{n=0}^m t^{2n+1}\frac{1+tV_n}{1+t^2V_n}
\prod_{i=0}^{n-1}
\frac{(1+v)(1+w)^i(1+z)V_i}{1+t^2V_i}+\cO(t^{2m+2})\\
&=\sum_{n=0}^m t^{2n+1}\frac{1+tV_n}{1+t^2V_n}(1+v)^n(1+z)^n(1+w)^{\binom{n}{2}}
\prod_{i=0}^{n-1}
\frac{V_i}{1+t^2V_i}+\cO(t^{2m+2}).
\end{align*}

Taking the limit as $m$ goes to infinity proves Theorem~\ref{thm-sym}.

\section{Self-Dual Orders Counted by Reduced Size}\label{sec-red}

So far, the main statistic of a Fishburn matrix has always been its size, i.e., 
the sum of its entries. However, for a self-dual Fishburn matrix $M$, we may
consider an alternative notion of size, which we call \emph{the reduced
size of $M$} and denote by $\rs(M)$, and which we define as the sum of all the
diagonal and south-east cells of the matrix. Note that the diagonal and
south-east cells determine a self-dual matrix uniquely. With the notation of the
previous section, $\rs(M)$ is equal to
$\dg(M)+\se(M)+\max(M)+\iso(M)$.

It turns out that the notion of reduced size is in some respects a more natural
statistic of self-dual interval orders than the notion of size that we have
used so far. In particular, we show that self-dual interval orders of a
given reduced size admit a matrix representation by a family of matrices that
have simpler structure than self-dual Fishburn matrices.

A \emph{row-Fishburn matrix} is an upper-triangular matrix of nonnegative
integers with the property that each row has at least one nonzero entry. A
row-Fishburn matrix is \emph{primitive} if all its entries are equal to 0 or~1.
As usual, the size of a row-Fishburn matrix is taken to be the sum of its
cells. The matrix statistics from Table~\ref{tab-stats} can be applied to
row-Fishburn matrices, even though the poset interpretations listed there are
not applicable. 

\begin{theorem}\label{thm-reduced}
 For any integers $n\ge 1$ and $m\ge 1$, the following three quantities are all
equal.
\begin{enumerate}
\item The number of self-dual Fishburn matrices of reduced size $n$, whose
last column has sum $m$, and whose diagonal cells are all equal to zero.
\item The number of self-dual Fishburn matrices of reduced size $n$, whose last
column has sum $m$, and whose diagonal cells are not all equal to zero.
\item The number of row-Fishburn matrices of size $n$ whose last column has
sum~$m$. 
\end{enumerate}
Moreover, for integers $n$, $m$, $p$ and $q$ with $p+q\ge
1$, the number of self-dual Fishburn matrices $M$ with $\rs(M)=n$, $\max(M)=m$,
$\dg(M)=p$ and $\iso(M)=q$ is equal to the number of row-Fishburn matrices $N$
of size $n$ with $\max(N)=m$, $\min(N)=p$ and $\iso(N)=q$. All these
relationships remain valid when restricted to primitive matrices.
\end{theorem}
\begin{proof}
We prove the statements by comparing the generating functions of the
relevant objects. We will first prove the statement for primitive matrices, the
statement for general matrices then follows easily.

Let $\rF$ be the set of primitive row-Fishburn matrices, and consider the
generating function
\[
 R(v,w,x,y)=\sum_{N\in\rF} v^{\max(N)} w^{\rest(N)}x^{\iso(N)}y^{\min(N)}.
\]
This generating function satisfies the identity
\[
 R(v,w,x,y)=\sum_{n\ge 0} \left( (1+x)(1+y)^n -1\right)\prod_{i=0}^{n-1} \left(
(1+v)(1+w)^i-1\right).
\]
To see this, note that the $n$-th summand on the right-hand side of this
expression is the generating function of primitive row-Fishburn matrices
with $(n+1)$ rows. Indeed, the factor $(1+x)(1+y)^n -1$ counts the number of
possibilities to fill the first row of such a matrix, while the factor
$(1+v)(1+w)^i-1$ counts the number of possibilities to fill the row $n+1-i$.

Recall that $\cS$ is the set of primitive self-dual Fishburn matrices, and let
$\cSn$ denote the set of primitive self-dual Fishburn matrices whose
north-east diagonal only contains zeroes, while $\cSd=\cS\setminus\cSn$ is the
set of primitive self-dual Fishburn matrices whose north-east diagonal contains
at least one positive cell. Let $S'(v,w,x,z)$ be the generating function
\[
 \Sr(v,w,x,z)=\sum_{M\in\cS} v^{\max(M)} w^{\se(M)} x^{\iso(M)} z^{\dg(M)},
\]
and let $\Srn(v,w)$ and $\Srd(v,w,x,z)$ be the analogous generating functions
for the sets $\cSn$ and $\cSd$, respectively, where $\Srn$ does not depend on
$x$ and $z$ because matrices from $\cSn$ have $\iso(M)=\dg(M)=0$. From
Theorem~\ref{thm-sym}, and in analogy to Lemma~\ref{lem-spp}, we see that
\begin{align*}
 \Sr(v,w,x,z)&=(1+x)\SP(1,v,w,z)-1\\
&=-1+(1+x)\sum_{n\ge 0}(1+z)^n(1+v)^n
(1+w)^{\binom{n}{2}}\prod_{i=0}^{n-1}\frac{V_i(v,w)}{1+V_i(v,w)}\\
&=-1+(1+x)\sum_{n\ge 0}(1+z)^n(1+v)^n
(1+w)^{\binom{n}{2}}\prod_{i=0}^{n-1}\frac{(1\!+\!v)(1\!+\!w)^i\!-\!1}{
(1\!+\!v)(1\!+\!w)^i } \\
&=-1\!+\!\sum_{n\ge 0}(1+x)(1+z)^n\prod_{i=0}^{n-1}\left((1+v)(1+w)^i-1\right).
\end{align*}
Since $\Srn(v,w)=\Sr(v,w,0,0)$ and $\Srd(v,w,x,z)=\Sr(v,w,x,z)-\Srn(v,w)$, we
get
\begin{align*}
 \Srn(v,w)&=-1+\sum_{n\ge0}\prod_{i=0}^{n-1}\left((1+v)(1+w)^i-1\right),\text{
and}\\
\Srd(v,w,x,z)&=\sum_{n\ge
0}\left((1+x)(1+z)^n-1\right)\prod_{i=0}^{n-1}\left((1+v)(1+w)^i-1\right)
\end{align*}
We see that $\Srn(v,w)=\Srd(v,w,v,w)=R(v,w,v,w)$, and that
$\Srd(v,w,x,z)=R(v,w,x,z)$, proving the theorem for primitive matrices. To
prove the non-primitive case, it is enough to observe that the generating
functions for general matrices may be obtained from the generating functions of
primitive matrices by substituting $\alpha/(1-\alpha)$ for each variable
$\alpha\in\{v,w,x,y,z\}$.
\end{proof}

Although Theorem~\ref{thm-reduced} follows easily from the generating function
formulas established before, it might still be worthwhile to provide a bijective
argument. Currently, we are not aware of such an argument. 

From Theorem~\ref{thm-reduced}, we may directly deduce the following corollary.
\begin{corollary}\label{cor-reduced}
 Let $s_m$ be the number of primitive self-dual interval orders of reduced size
$m$ and let $r_m$ be the number of primitive row-Fishburn matrices of size $m$.
Then $s_m=2r_m$ for each $m\ge 1$, and
\[
 \sum_{m\ge 1} r_m x^m = \sum_{n\ge 0}
\prod_{i=0}^{n}\left((1+x)^{i+1}-1\right).
\]
Let $t_m$ be the number of self-dual interval orders of reduced size $m$, and
let $q_m$ be the number of row-Fishburn matrices of size~$m$. Then $t_m=2q_m$
for each $m\ge 1$, and 
\[
 \sum_{m\ge 1} q_m x^m =\sum_{n\ge 0} \prod_{i=0}^{n} \left(
\frac{1}{(1-x)^{i+1}}-1\right).
\]
\end{corollary}
Let us remark that the numbers $(r_m)_{m\ge 1}$ from Corollary~\ref{cor-reduced}
correspond to the sequence A179525 in OEIS~\cite{oeis}, while $(q_m)_{m\ge 1}$ 
conjecturally correspond to A158691. To be more precise, A158691 is the sequence
of coefficients of the power series $\sum_{n\ge
0}\prod_{i=1}^{n}\left(1-(1-x)^{2i-1}\right)$, which, according to the notes
in the OEIS entry, are conjectured to be equal to $q_m$ for $m\ge 1$.

\section{Final Remarks and Open Problems}
The formulas of the form we derived in this paper provide an efficient way to
explicitly compute the coefficients of the corresponding generating functions. 
They are also occasionally useful in establishing correspondences between
different combinatorial structures, as shown in
Theorem~\ref{thm-reduced}. It is not clear, however, whether one can
use such formulas to extract information about the asymptotic growth of the
coefficients. Zagier~\cite{Zagier} has used formula~\eqref{eq-zag}, together
with several non-trivial power series identities, to find a very precise
asymptotic estimate of the number of interval orders on $n$ elements. We state
a weaker version of this estimate as fact.
\begin{fact}[\cite{Zagier}]\label{fac-gn} If $g_n$ is the number of interval
orders of size $n$, then
\[
 g_n=(\alpha+\cO(n^{-1})) n! \left(\frac{6}{\pi^2}\right)^n\sqrt{n}\quad \text{
  with } \alpha=\frac{12\sqrt{3}}{\pi^{-5/2}}e^{\pi^2/12}.
\]
\end{fact}
Drmota~\cite{drmota} has pointed out that from this estimate, we may deduce
the asymptotic fraction of primitive posets among all interval orders. 
\begin{fact}[\cite{drmota}]
 With $g_n$ as above, and with $p_n$ being the number of primitive interval
orders of size $n$, we have 
\[
\lim_{n\to\infty} \frac{p_n}{g_n} = e^{-\pi^2/6}.
\]
\end{fact}
\begin{proof}
The generating functions $F(x)=\sum_{m\ge0} g_mx^m$ and $G(x)=\sum_{n\ge 0} p_n
x^n$ are related by $F(x)=G(x/(1-x))$, or equivalently, $G(x)=F(x/(1+x))$.
Thus, for every $n\ge1$, we have
\begin{align*}
 p_n&= [x^n] F\left(\frac{x}{1+x}\right)=[x^n]\sum_{m=1}^n g_m
\left(\frac{x}{1+x}\right)^m
=\sum_{m=1}^n g_m (-1)^{n-m} \binom{n-1}{m-1}\\
&= \sum_{m=1}^n m!\left(\frac{6}{\pi^2}\right)^m
\sqrt{m}(-1)^{n-m}\frac{(n-1)!}{(m-1)!(n-m)!}\left(\alpha+\cO\left(\frac{1}{m}
\right)\right)\\
&= \alpha n!\left(\frac{6}{\pi^2}\right)^n\sqrt{n}  \sum_{m=1}^n
 \left(\frac{m}{n}\right)^{3/2}(-1)^{n-m} \left(\frac{6}{\pi^2}\right)^{m-n}
\frac{1}{(n-m)!}\left(1+\cO\left(\frac{1}{m}\right)\right)\\
&=g_n\sum_{k=0}^{n-1}
\left(1-\frac{k}{n}\right)^{3/2}\frac{(-\pi^2/6)^k}{k!}\left(1+\cO\left(\frac{1}
{n-k}\right)\right)=g_n e^{-\pi^2/6} (1+o(1)).\qedhere
\end{align*}
\end{proof}

Our main general open problem is to obtain similar asymptotic estimates for the
enumeration of self-dual interval orders, counted either by their size or their
reduced size. Using the generating function formulas, we can easily enumerate
self-dual interval orders of a given size, and using numerical manipulations
similar to those described by Zagier~\cite[Section 3]{Zagier} to accelerate
convergence, we can then make conjectures about the coefficient asymptotics.

\begin{conjecture}\label{con-red}
Let $s_n$ be the number of primitive self-dual interval orders of reduced size
$n$, and let $t_n$ be the number of self-dual interval orders of reduced
size~$n$. Then
%\[
\begin{gather*}
t_n=(\beta+\cO(n^{-1})) n!\left(\frac{12}{\pi^2}\right)^n\text{
with } \beta=\frac{12\sqrt{2}}{\pi^2}e^{\pi^2/24},\\
% \]
% \[
 \text{ and }\lim_{n\to\infty} \frac{s_n}{t_n}=e^{-\pi^2/12}.
% \]
\end{gather*}
\end{conjecture}

\begin{conjecture}\label{con-sym}
Let $r_n$ be the number of self-dual interval orders of size $n$, and let $q_n$
be the number of primitive self-dual interval orders of size~$n$. Then
\begin{gather*}
 r_n\!=\!(\gamma\!+\!\cO(n^{-1/2}))\sqrt{n} \left(\frac{\delta
n}{e}\right)^{\!n/2} 2^{\sqrt{\delta
n}}\text{  with }\gamma\!\approx\!1.361951039\dotsc \text{ and }
\delta\!=\!\frac{6}{\pi^2},\\
 \text{ and }\lim_{n\to\infty} \frac{q_n}{r_n} = \frac{1}{2} e^{-\pi^2/12}.
\end{gather*}
\end{conjecture}

In a similar vein, one can ask whether the multi-variate generating functions
can provide information about the distribution of the relevant statistics
within the set of interval orders of a given size, or within the set of Fishburn
matrices of a given dimension. Here are two examples of the kind of questions
that arise.

\begin{problem} What is the average sum of entries in a primitive
$k\times k$ Fishburn matrix?
\end{problem}
\begin{problem}
What is the average number of minimal elements in an $n$-element interval order?
\end{problem}

\section*{Acknowledgement}
I am indebted to Michael Drmota for his insightful remarks and helpful
suggestions.
\small
\bibliographystyle{plain}
\bibliography{intord}

\end{document}